\documentclass[preprint,12pt,times]{article}
\usepackage{graphicx}
\usepackage{amsthm}
\usepackage{amssymb}
\usepackage{amsmath}
\usepackage{amscd}
\usepackage{bbm}
\usepackage{float}
\usepackage{booktabs}
\numberwithin{equation}{section}

\newtheorem{theorem}{Theorem}[section]

\newtheorem{remark}{Remark}[section]

\usepackage[numbers,sort&compress]{natbib}

\numberwithin{figure}{section}
\numberwithin{table}{section}

\allowdisplaybreaks[2]

\usepackage{pgfplots}
\newcommand\btd{\raise 2pt \hbox{$\hat\bigtriangledown$}\hskip 1.5pt}
\newcommand\bt{\raise 2pt \hbox{$\bigtriangledown$}\hskip 1.5pt}

\newcommand{\ud}{\mathrm{d}}

\begin{document}
\date{}
\title{Kinetic energy properties of irrotational deep-water Stokes waves}
\author{Jian Li$^a$, ~Shaojie Yang$^{a,b}$\thanks{Corresponding author: jianli\_jakura@163.com(Jian Li); shaojieyang@kust.edu.cn (Shaojie Yang)} \\~\\
\small$^a$ Department of  Mathematics,~~Kunming University of Science and Technology,  \\
\small Kunming, Yunnan 650500, China\\
\small$^b$ Research Center for Mathematics and Interdisciplinary Sciences,\\
 \small Kunming University of Science and Technology,\\
\small Kunming, Yunnan 650500, China}

\date{}
\maketitle
\begin{abstract}
We investigate  kinetic energy properties of an irrotational deep-water Stokes wave. By introducing a conformal hodograph transform, we  perform complex analysis in the new variables, and show that the existence of the streamline time-period for deep-water Stokes waves. Relying on the conformal mappings, we prove some qualitative results about the kinetic energy such as the convexity property of integral means of kinetic energy, the streamline time-period being  non-increasing and independent of initial data and the monotonicity and logarithmic convexity of the total kinetic energy over a streamline time-period. Moreover, taking advantage of the non-increasing nature of the streamline time-period, we show that the streamline time-period strictly larger than the wave period, which suggests the fact that there are no closed paths for all particles in a deep-water Stokes flow and the drift of any streamline is positive. These results allow us to quantify the observation that the kinetic energy and the streamline time-period are larger near the free  surface and decrease with increasing the depth.\\


\end{abstract}
\noindent\rule{16.5cm}{0.5pt}

\section{Introduction}
In hydrodynamics, the study of steady deep-water waves is a hot topic because typical near-surface water flows are analyzed in terms of the superposition and interaction of these waves, while at great depths the water is almost at rest. This interpretation relies on a detailed understanding of the dynamics of these steady deep-water waves, for which no closed-form solution is known.
In 1847, Stokes proposed a perturbation procedure that provided the first successive approximations to the flow beneath irrotational periodic traveling deep-water waves (called  deep-water "Stokes wave"), and computers have been used to obtain explicitly higher-order Stokes approximations  in recent decades. However, while the power series method proved that Stokes expansion converges for waves of very small steepness \cite{r6}, it has now been established that even for waves of moderate steepness, truncation by a certain order yields inaccurate results \cite{r11,r12}. Relying on analytical methods, a lot of works has been studied about the dynamics of deep-water Stokes wave \cite{r13,r5,r9, r14,r15,r16,r17}.

\subsection{The governing equations}
For two-dimensional water waves, it suffices to investigate the flow characteristics in a
cross-section oriented towards the direction of wave propagation. We choose Cartesian
coordinates $(x,y)$ with the $x$-axis pointing in the direction of wave propagation and the
$y$-axis oriented upwards. Let $y=\eta(x,t)$ be the free surface, $(u(x,y,t),v(x,y,t))$ the velocity field, with the water occupying
the region $\{(x,y ): -\infty <y <\eta(x, t)\}$ at time $t$. Under the physically reasonable assumption of a homogeneous inviscid fluid, the governing equations for two-dimensional irrotational water waves are the Euler equations:
\begin{equation}\label{1.1}
\begin{cases}
\begin{aligned}
&u_{t} + u u_{x} + v u_{y} &&=  P_x \quad -\infty <y < \eta(x,t), \\
&v_{t} + u v_{x} + v v_{y} &&=  P_y-g \quad -\infty <y < \eta(x,t).\\
\end{aligned}
\end{cases}
\end{equation}
Here $t$ represents time, $P$ is the pressure and $g$=9.8$m/s^{2}$ is the constant gravitational acceleration on the surface of the Earth. Moreover, the associated boundary conditions are the kinematic boundary condition
\begin{equation}\label{1.2}
v=\eta_{t}+u \eta_x \quad\text{on   } y=\eta(x,t),
\end{equation}
which expresses the fact that the water's free surface is an interface.
The pressure of the fluid corresponds to the atmospheric pressure $P_{atm}$ at sea level, that is the dynamic boundary condition
\begin{equation}\label{1.3}
P= P_{atm} \quad \text{on    } y = \eta(x,t),
\end{equation}
which decouples the motion of the water from that of the air above it.
Since the absence of underlying currents is ensured by the irrotational characteristic of the flow, then we have
\begin{equation}\label{1.4}
u_y-v_x=0 \quad \text{for   } -\infty <y < \eta(x,t).
\end{equation}
Since the fact that the density of water is constant, the equation of mass conservation takes the form
\begin{equation}\label{1.5}
u_x+v_y=0 \quad \text{for   } -\infty <y <\eta(x,t).
\end{equation}
Moreover, since the wave speed is much larger than the horizontal fluid velocity (see \cite{r6}), that is
\begin{equation}\label{1.6}
u(x,y)< c\quad \quad \text{for all  } -\infty <y \leq\eta(x,t).
\end{equation}
We assume the water is practically at rest at great depths, implying the constraint
\begin{equation}\label{1.7}
(u,v)\rightarrow (0,0)\quad as \quad y\rightarrow -\infty \quad \text{uniformly for} (x,t)\in {\mathbb R}^2.
\end{equation}
Taking advantage of the $(x,t)$-dependence of the form $(x-ct)$, passing to the moving frame, we reformulate the governing equations $\eqref{1.1}$-$\eqref{1.7}$ as
\begin{equation}\label{1.8}
\begin{cases}
(u-c) u_{x} + v u_{y} = - P_x \quad -\infty <y < \eta(x,t), \\
(u-c) v_{x} + v v_{y} = -P_y-g \quad -\infty <y < \eta(x,t),\\
u_x+v_y=0 \quad \text{for   } -\infty <y <\eta(x,t),\\
u_y-v_x=0 \quad \text{for   } -\infty <y < \eta(x,t),\\
v=(u-c)\eta_x \quad\text{on   } y=\eta(x,t),\\
P= P_{atm} \quad \text{on    } y = \eta(x,t),\\
(u,v)\rightarrow (0,0)\quad as \quad y\rightarrow -\infty \quad \text{uniformly for} (x,t)\in {\mathbb R}^2,
\end{cases}
\end{equation}
with
\begin{equation}\label{a1}
u(x,y)< c\quad \quad \text{for all  } -\infty <y \leq\eta(x,t),
\end{equation}
where $c>0$ represents wave speed.
A Stokes wave is traveling wave solution to the governing equations $\eqref{1.1}$-$\eqref{1.7}$ for which there is a period $\lambda$ such that the free surface $\eta$ and the velocity field $(u,v)$ have period $\lambda$ in the $x$-variable, and $\eta$, $u$ and $P$ are symmetric about the wave crest.

Without loss of generality, we assume the wave crest to be located at $x=0$, then $x=\pm \lambda/2$ is trough lines and $\eta(x)=\eta(-x)$ for all $x\in\mathbb R$ with $\eta$ increasing on $[-\lambda/2,0]$ and decreasing on $(0,\lambda/2]$ (See Figure \ref{f1}). Moreover, the profile is strictly monotonic between successive crests and troughs, and there is a single crest and trough per wavelength $\lambda$.
Passing to the moving frame, we assume that the velocity field $(u,v)$ is bounded and continuously differentiable throughout the fluid domain
\begin{equation*}
D=\{(x,y)\in\mathbb R^2: -\infty <y<\eta(x)\}.
\end{equation*}
\subsection{Stream function, velocity potential and  hodograph transform}

We introduce the \emph{stream function} $\psi(x,y)$, defined up to an additive constant by
\begin{align}\label{1.9}
\psi_y=u-c,\qquad \psi_x=-v\quad \text{for }-\infty<y<\eta(x).
\end{align}
From $\eqref{1.4}$-$\eqref{1.5}$, we can infer that stream function $\psi$ is harmonic in the fluid domain $D$. Furthermore, $\psi$ is a constant on the free surface $y=\eta(x)$. Without loss of generality, we assume $\psi=0$ on $y=\eta(x)$.
Then $\eqref{1.8}$ can be rewritten as
\begin{equation}\label{1.10}
\begin{cases}
\psi_y \psi_{xy} -\psi_x\psi_{yy} = - P_x \quad -\infty <y < \eta(x), \\
-\psi_y  v_{xx} +\psi_x \psi_{xy} = -P_y-g \quad -\infty <y < \eta(x),\\
\triangle \psi=0 \quad \text{for   } -\infty <y <\eta(x),\\
\psi=0 \quad\text{on   } y=\eta(x),\\
P= P_{atm} \quad \text{on    } y = \eta(x),\\
\nabla \psi(x,y)\rightarrow (0,-c)\quad as \quad y\rightarrow -\infty \quad \text{uniformly for} (x,t)\in {\mathbb R}^2.
\end{cases}
\end{equation}
Indeed, the stream function $\psi$ has the properties of periodicity and symmetry(see\cite{r7}), that is
\begin{equation*}
\psi(x+\lambda,y)=\psi(x,y),\quad \psi(x,y)=\psi(-x,y)\quad \text{for } -\infty <y<\eta(x).
\end{equation*}
Moreover, taking advantage of $\eqref{1.9}$, we obtain
\begin{equation}\label{1.11}
\begin{cases}
u(x,y)=u(-x,y) \quad \text{for } -\infty <y\leq\eta(x),\\
v(x,y)=-v(-x,y)\quad \text{for } -\infty <y\leq\eta(x),
\end{cases}
\end{equation}
which shows that the horizonal velocity $u$ is an even function, and $v$ is an odd function in the $x$-variable. Moreover, the above argument expresses the fact that these properties have to hold true if the traveling wave profile is symmetric. By $\eqref{1.11}$, we obtain $v(\cdot,y)=0$. Furthermore, due to the periodicity of $v$, we have $v(\pm \lambda/2,y)=0$. $\eqref{1.6}$ and $\eqref{1.8}$ together with the monotonicity of $\eta$, then we have
\begin{equation}\label{uj}
v(\cdot,y)\leq0 \text{ for }x\in[-\lambda/2,0]\text{ and } v(\cdot,y)\geq0 \text{ for }x\in(0,\lambda/2].
\end{equation}
We now define the \emph{velocity potential} $\varphi(x,y)$ as the harmonic conjugate of $\psi(x,y)$, defined up to an additive constant by
\begin{equation*}
\varphi_x=\psi_y,\quad \varphi_y=-\psi_x,\quad -\infty <y<\eta(x).
\end{equation*}
According to Lagrange Mean-Value theorem, for some $\delta_{x,y}\in(x,x+\lambda)$, we find that the formula holds true throughout the whole fluid domain $D$, that is
\begin{equation*}
\varphi(x+\lambda,y)-\varphi(x,y)=\lambda\varphi_x(\delta_{x,y}),y)=\lambda(u(\delta_{x,y},y)-c)<0.
\end{equation*}
Therefore, $\varphi$ is not a periodic function in the $x$-variable. Due to the periodicity of $\psi$, we consider the periodic domain
\begin{equation*}
D_{\lambda}=\{(x,y)\in\mathbb R^2: -\lambda/2<x < \lambda/2, -\infty<y<\eta(x)\}.
\end{equation*}
Since $u(x,y)<c$ holds true throughout the closure of the fluid domain $D$,  the function $x\mapsto \varphi(x,y)$ attains its maximum $\varphi_{\max}$ at $x=-\lambda/2$, and the minimum $\varphi_{\min}$ of which is attained at $x=\lambda/2$. The function $\varphi$ is harmonic in the fluid domain $D$, implying the fact that $\varphi_{\max}$ and $\varphi_{\min}$ are equal to a constant at the boundary. Without loss of generality, we let $\varphi_{\min}=0$. We consider the hodograph transform $h(z)=\varphi(x,y)+i\psi(x,y)$ with $z=x+iy$ is analytic in the interior of fluid domain $D$. $h'(z)=u-c+i(-v)$ is analytic in $D$, and by $\eqref{1.7}$ we can deduce
\begin{equation*}
\underset{y\rightarrow -\infty}\lim \underset{x\in\mathbb R} \sup|h'(x,y)|=|c|.
\end{equation*}
Due to the periodicity of $\psi$ in the $x$-variable, we obtain
\begin{equation*}
h(z+\lambda)-h(z)=\varphi(x+\lambda,y)-\varphi(x,y),
\end{equation*}
which is a constant.
Let $y\rightarrow -\infty$, in light of $\eqref{1.7}$, we have
\begin{equation*}
h(z+\lambda)-h(z)=\int_{z}^{z+\lambda} h'(w) \,\ud w=\int_{z}^{z+\lambda} (u-c+i(-v)) \,\ud w\rightarrow -c\lambda.
\end{equation*}
Therefore, we find that
\begin{equation*}
\varphi(x+\lambda,y)-\varphi(x,y)=-c\lambda,
\end{equation*}
which holds true for any $z=x+iy$ in the fluid domain $D$.
Then we have
\begin{equation*}
\varphi_{\min}-\varphi_{\max}=-\varphi_{\max}=\int_{-\frac{L}{2}}^{\frac{L}{2}}\varphi_x\,\ud x=-c\lambda.
\end{equation*}
Therefore, $\varphi_{\max}=c\lambda$.
From $\eqref{1.6}$, we obtain that the minimum of $\psi$ is attained at the free surface $z=\eta(x)$, and is equal to zero. Furthermore, we derive that
\begin{equation*}
-\psi_{\max}=\int_{-\infty}^{\eta(x)}\psi_y(x,y)\,\ud y=-\infty,
\end{equation*}
that is $\psi_{\max}=+\infty$.
Notice that the hodograph transform $h=\varphi+i\psi$ is a biholomorphic function mapping the periodic fluid domain $D_{\lambda}$ onto the open half rectangle domain
\begin{equation*}
\Omega_{\lambda}=\{(q,p)\in{\mathbb R^2}: 0 <q <c\lambda, 0 < p < +\infty\}.
\end{equation*}

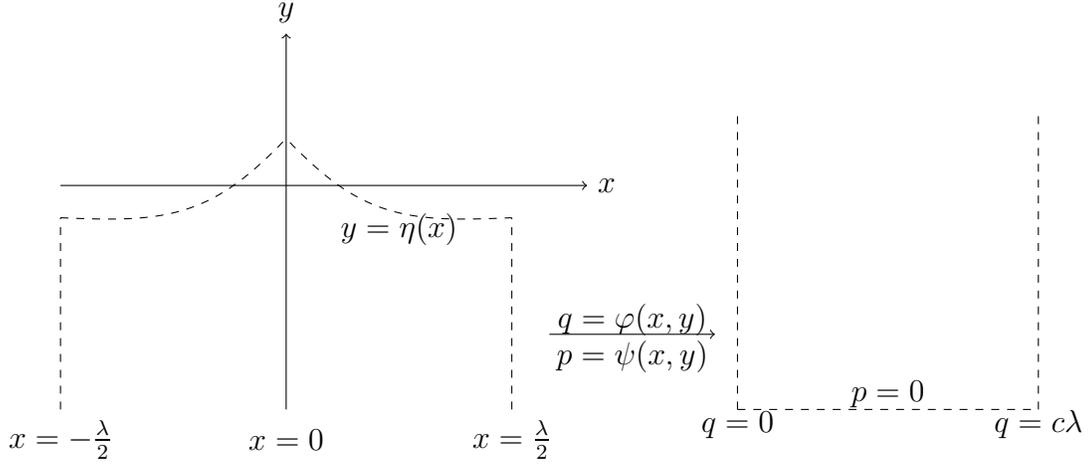
\begin{figure}[H]
\centering
\begin{tikzpicture}
    \draw[->] (-3,0.98) -- (4,0.98) node[right] {$x$}; 
    \draw[->] (0,-2) -- (0,3) node[above] {$y$}; 

    \draw[domain=-3:3,smooth,variable=\x,black,samples=100,dashed] plot ({\x},{(exp(-abs(\x))*cos(deg(-abs(\x))))+0.6});

    \draw [dashed](-3,-2) --(-3,{(exp(-abs(-3))*cos(deg(-abs(-3))))+0.6});
    \draw [dashed](3,-2) -- (3,{(exp(-abs(3))*cos(deg(-abs(3))))+0.6});
    \node at (1.5,0.4) {$y = \eta(x)$};
    \node at (-3,-2.4) {$x=-\frac{\lambda}{2}$};
    \node at (3,-2.4) {$x=\frac{\lambda}{2}$};
    \node at (0,-2.4) {$x=0$};

    \draw [dashed](6,-2) -- (6,2); 
    \draw [dashed](10,-2) -- (10,2); 
    \draw[->] (3.5,-1) -- (5.7,-1);

    \draw [dashed](6,-2) -- (10,-2);
    \node at (6,-2.2) {$q=0$};
    \node at (10,-2.2) {$q=c\lambda$};
    \node at (8,-1.8) {$p=0$};
    \node at (4.6,-0.8) {$q=\varphi(x,y)$};
    \node at (4.6,-1.3) {$p=\psi(x,y)$};
\end{tikzpicture}
\caption{The conformal hodograph transform $h$ maps the fluid domain in the moving frame into an open half domain in the $\left(q, p\right)$-plane.}\label{f1}
\end{figure}

Considering the open half rectangle domain
\begin{equation*}
\Omega=\{(q,p): q\in \mathbb R, 0 < p < +\infty\},
\end{equation*}
and the conformal bijection $q+ip\mapsto x+iy$ from $\Omega$ to $D$, then there exists an inverse transform:
\begin{equation}\label{2.15}
\begin{cases}
x=x(q,p),\\
y=y(q,p),
\end{cases}
\end{equation}
and has the ``periodicity'' properties(see Figure \ref{f1})
\begin{equation}\label{2.17}
\begin{cases}
x(q+c\lambda,p)=x(q,p)-\lambda\quad \text{for} \quad(p,q)\in \Omega,\\
y(q+c\lambda,p)=y(q,p)\quad \text{for} \quad (p,q)\in \Omega,\\
\underset{p\rightarrow+\infty}\lim y(q,p)=-\infty\quad \text{ for }\quad q\in\mathbb R.
\end{cases}
\end{equation}
The aim of this paper is to investigate the kinetic energy  properties of irrotational deep-water Stokes waves. Taking advantage of complex analysis approach, we obtain several properties of streamlines and kinetic energy characteristics, including the integral means of the kinetic energy, the streamline time-period, the total kinetic energy of a fluid particle over a streamline time-period. The outline of this paper is as follows.

In Section \ref{sec2}, we derive properties of the integral means of kinetic energy. A crucial step in our method is to observe that the integral of kinetic energy can be expressed as an integral over a horizontal segment of the modulus of an appropriately defined analytic function. This allows us to apply Hardy's convexity theorem to a unit disk. We show the convexity and logarithmic convexity properties of the integral of kinetic energy, as well as its non-increasing nature. In Ref.\cite{r8}, the kinetic energy index $s\in[1,+\infty)$, we extend the kinetic energy index to $s\in (-\infty,-1/2]\cup [1/2,+\infty)$, which can provide theoretical basis for numerical simulations of the kinetic energy for irrotational deep-water waves.

In Section \ref{sec3}, we derive qualitative results about the streamline time-period, which refers to  a particle's time taken to repeat its trajectory. We establish that streamline time-period is independent of initial data based on ideas from Ref.\cite{r4}. Additionally, we show that the streamline time-period depends solely on the streamline. According to Theorem $\ref{Theorem3.1}$, we prove convexity and logarithmic convexity properties of the streamline time-period for a fluid particle, as well as its non-increasing nature. Moreover, taking advantage of the properties of deep-water Stokes waves, we provide upper and lower bounds of the streamline time-period. We find that the streamline time-period is strictly larger than the wave period. Furthermore, there are no closed paths for all fluid particles, and the drift of any streamline is positive.

In Section \ref{sec4}, we obtain properties of the total kinetic energy and the total kinetic energy passing to moving frame for a fluid particle over a streamline time-period. We derive that the total kinetic energy over a streamline time-period is independent of initial location, then we prove the convexity and logarithmic convexity properties of total kinetic energy over a streamline time-period, as well as its non-increasing nature, and total kinetic energy passing to the moving frame is equal to a constant. Moreover, we provide a better understanding of the kinetic energy for irrotational deep-water waves.

\section{The integral means of kinetic energy}\label{sec2}
In this section, we investigate some qualitative results about integral means of kinetic energy passing to the moving frame. Let $E=((u-c)^2+v^2)/2$ be the kinetic energy at different depths passing to the moving frame. Moreover, since $h(x,y)=(q,p)$, then $(x,y)=h^{-1}(q,p)$. Define
\begin{align}\label{3.1}
\mu_s(E,p) = \frac{1}{c\lambda} \int_0^{c\lambda} E(h^{-1}(q,p))^s\,\ud q.
\end{align}
\begin{theorem}\label{Theorem3.1}
Suppose $p>0$ and $s\in{(-\infty,-1/2]\cup[1/2,+\infty)}$. The function $\mu_s(E,p)$ is convex and non-increasing, and $\log\mu_s(E,p)$ is a convex function.
\end{theorem}
\begin{proof}
Let $R_0=(0,c\lambda)\times(0,+\infty)$, and $k=2\pi/{c\lambda}$. From $\eqref{1.10}$, then $h=q+ip=\varphi+i\psi$ is analytic and smooth in the interior of $D$, and maps the fluid domain $D$ onto the open half rectangle domain $\Omega$. The function $h'=\varphi_x+i\psi_x=u-c+i(-v)$ is analytic and periodic, and $|h'(h^{-1}(q,p))|^2=2E(h^{-1}(q,p))$. Owing to the periodicity of $h'$, we can transfer our analysis from the periodic fluid field $D_{\lambda}$ to the periodic rectangle domain $\Omega_{\lambda}$. Considering the conformal diffeomorphism:
$q+ip \mapsto \mathrm{e}^{ik(q+ip)},$
which maps the periodic rectangle domain $\Omega_{\lambda}$ onto a unit disk
\begin{align*}
\mathcal{S} = \{z \in \mathbb{C}:  |z| < 1\}.
\end{align*}
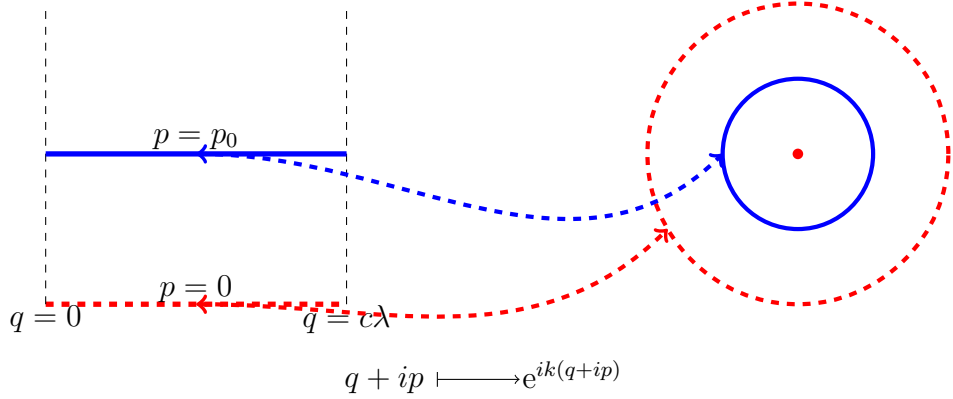
\begin{figure}[H]
\centering
\begin{tikzpicture}
    \draw[line width=2pt, red,dashed] (-10,-2) -- (-6,-2);

    \draw [dashed](-10,-2) -- (-10,2);
    \draw [dashed](-6,-2) -- (-6,2);
    \draw[line width=2pt, blue] (-10,0) -- (-6,0);
    \node at (-10,-2.2) {$q=0$};
    \node at (-6,-2.2) {$q=c\lambda$};
    \node at (-8,-1.8) {$p=0$};
    \node at (-8,0.2) {$p=p_0$};
    \node at (-5.5,-3) {$q+ip$};
    \node at (-3,-2.93) {$\mathrm{e}^{ik(q+ip)}$};
    \fill[red] (0,0) circle (2pt);
    \draw [dashed,red, ultra thick](0,0) circle (2);
    \draw [blue, ultra thick](0,0) circle (1);
    \draw[<->,blue,ultra thick,dashed] (-8,0) to[out=0,in=225] (-1,0);
    \draw[<->,red,ultra thick,dashed] (-8,-2) to[out=0,in=225] (-1.732,-1);
    \draw[|->] (-4.8,-3) -- (-3.7,-3);
\end{tikzpicture}

\caption{The conformal diffeomorphism maps the periodic rectangle domain onto the unit disk.}
\end{figure}

Note that there exists an inverse map: $z\mapsto-\frac{i}{k}\log z$. We consider the map $G$: $\mathcal S \rightarrow \mathbb {C}$, given by
\begin{align}\label{DF}
G(z)=(h'\circ h^{-1})(-\frac{i}{k}\log z),
\end{align}
which can extend to a continuous function on $\mathcal{S}$. The Morera's theorem ensures that the function $G(z)$ is analytic in the interior of $\mathcal S$.
Setting $G\in H^{2s}(\mathcal S)$, where $H^{2s}(\mathcal S)$ is the Hardy space, then
\begin{align*}
M_{2s}^{2s}(G,r) &= \frac{1}{2\pi}\int_0^{2\pi} |G(r\mathrm{e}^{i\theta})|^{2s} \,\mathrm{d}\theta  \\
&= \frac{1}{2\pi}\int_0^{2\pi} |(h'\circ h^{-1})\left(-\frac{i}{k}\log(r\mathrm{e}^{i\theta})\right)|^{2s} \,\ud \theta  \\
&= \frac{1}{2\pi}\int_0^{2\pi} \left|(h'\circ h^{-1})\left(\frac{\theta}{k}-i\frac{\log r}{k}\right)\right|^{2s} \,\ud\theta  \\
&= \frac{1}{2\pi}\int_0^{2\pi} \left|(h'\circ h^{-1})\left(\frac{\theta}{k},-\frac{\log r}{k}\right)\right|^{2s} \,\ud \theta.
\end{align*}
Substituting $\gamma=\frac{\theta}{k}$, we have
\begin{align*}
M_{2s}^{2s}(G,r)&=\frac{k}{2\pi}\int_0^{\frac{2\pi}{k}} \left|h'(h^{-1}(\gamma,-\frac{\log r}{k}))\right|^{2s} \,\ud \gamma\\
&=\frac{2^s}{c\lambda}\int_0^{c\lambda}\left |E(h^{-1}(\gamma,-\frac{\log r}{k}))\right|^s \,\ud \gamma\\
&=2^{s}\mu_{s}(E,-\frac{\log r}{k}).
\end{align*}
Therefore, let $r=\mathrm{e}^{-kp}$, we obtain
\begin{align}\label{3.2}
\mu_s(E,p)=\frac{1}{2^s} M_{2s}^{2s}(G,\mathrm{e}^{-kp}).
\end{align}
Furthermore,
\begin{equation}\label{W}
\log \mu_s(E,p)=2s\log M_{2s}(G,\mathrm{e}^{-kp})-s\log2.
\end{equation}
According to the Hardy's convexity theorem for analytic function in the unit disk (see Theorem 1.5 in \cite{r1}), for any $2s\geq0$ and
$r<1$, $M_{2s}(G,r)$ is a non-decreasing function of $r$, and $\log M_{2s}(G,r)$ is a convex function of $\log r$, i.e. for any $r_1$, $r_2$ and every $\alpha \in [0,1]$, we have
\begin{align*}
M_{2s}(G,r) \leq [M_{2s}(G,r_1)]^{\alpha} [M_{2s}(G,r_2)]^{1-\alpha},
\end{align*}
where $\log r =\alpha\log r_1 +(1-\alpha)\log r_2$. By $\eqref{W}$, $\log \mu_s(E,p)$ is a convex function. We let
\begin{align*}
\sigma(r)=\ln M_{2s}(G,r)\Rightarrow M_{2s}(G,r)=\mathrm{e}^{\sigma(r)}.
\end{align*}
Since $\sigma(r)$ is a convex function, then
\begin{align}\label{3.4}
\frac{\ud^2}{\ud r^2}M_{2s}(G,r)=\mathrm{e}^{\sigma(r)}(\sigma'^2(r)+\sigma''(r))\geq 0.
\end{align}
Therefore, the convexity of the map $r \mapsto \log M_{2s}(G,r)$ implies the convexity of the map $r \mapsto M_{2s}(G,r)$.
Moreover, since the function $G(z)$ is analytic on $\mathcal S$, then
\begin{equation*}
G(z)=\sum_{n=0}^{\infty}a_n z^n, \quad |z|<1.
\end{equation*}
Furthermore, we get
\begin{equation*}
M_{2s}^{2s}(G,r)=\sum_{n=0}^{\infty}|a_n|^{2s}r^{2sn},\quad r<1.
\end{equation*}
Next, taking advantage of the convex and non-decreasing properties of the map $r\mapsto M_{2s}(G,r)$, we will show that the map $p\mapsto M_{2s}^{2s}(G,\mathrm{e}^{-kp})$ is convex and non-increasing.
 Let $\mathcal F(r)=M_{2s}(G,r)$, then we obtain
\begin{align*}
\frac{\ud}{\ud p}\mathcal F(\mathrm{e}^{-kp})=-k\mathrm{e}^{-kp}\mathcal F'(\mathrm{e}^{-kp})\leq0
\end{align*}
and
\begin{align*}
\frac{\ud^2}{\ud p^2}\mathcal F(\mathrm{e}^{-kp})=k^2\mathrm{e}^{-kp}\left( \mathrm{e}^{-kp}\mathcal F''(\mathrm{e}^{-kp})+\mathcal F'(\mathrm{e}^{-kp})\right)\geq 0.
\end{align*}
Therefore, we find that the convexity of the map $r \mapsto \ M_{2s}(G,r)$ implies the convexity of the map $p \mapsto M_{2s}(G,\mathrm{e}^{-kp})$, and the map $p\mapsto M_{2s}(G,\mathrm{e}^{-kp})$ is non-increasing. Note that the map $r\mapsto\ M_{2s}^{2s}(G,\mathrm{e}^{-kp})$ is convex and non-increasing for $s\geq\frac{1}{2}$, that is
\begin{align*}
\frac{\ud}{\ud p}M_{2s}^{2s}(G,\mathrm{e}^{-kp})=\frac{\ud}{\ud p}\mathcal F^{2s}(\mathrm{e}^{-kp})=-2ks\mathrm{e}^{-kp}\mathcal F'(\mathrm{e}^{-kp})\mathcal F^{2s-1}(\mathrm{e}^{-kp})\leq 0
\end{align*}
and
\begin{equation*}
\begin{split}
     \frac{\ud^2}{\ud p^2}M_{2s}^{2s}(G,\mathrm{e}^{-kp})&=\frac{\ud^2}{\ud p^2}\mathcal F^{2s}(\mathrm{e}^{-kp})=2sk^2\mathrm{e}^{-kp} \left[ \mathrm{e}^{-kp}\big((2s-1)\mathcal F^{2s-2}(\mathrm{e}^{-kp})\mathcal F'^{2}(\mathrm{e}^{-kp}) \right. \\
    & \qquad \left. + \mathcal F^{2s-1}(\mathrm{e}^{-kp})\mathcal F''(\mathrm{e}^{-kp})\big)+\mathcal F'(\mathrm{e}^{-kp})\mathcal F^{2s-1}(\mathrm{e}^{-kp}) \right] \geq 0.
\end{split}
\end{equation*}
From $\eqref{3.2}$, for $ p\in(0,+\infty)$ and $s\geq\frac{1}{2}$, $p \mapsto \mu_{s}(E,p)$ is a non-increasing and convex function. By $\eqref{a1}$, then we have $h'\neq0$ on $D$.

For $s\leq-\frac{1}{2}$, we obtain
\begin{align}
\mu_s(E,p)&= \frac{1}{c\lambda} \int_0^{c\lambda} \left( E(h^{-1}(q,p))\right)^s\,\ud q \nonumber \\
&=\frac{1}{c\lambda} \int_0^{c\lambda} \left(\frac{1}{E(h^{-1}(q,p))}\right)^{-s}\,\ud q \nonumber \\
&=\frac{1}{2^s c\lambda} \int_0^{c\lambda} \left(\frac{1}{(u(h^{-1}(q,p))-c)^2+(v(h^{-1}(q,p)))^2}\right)^{-2s}\,\ud q \nonumber\\
&=\frac{1}{2^s c\lambda} \int_0^{c\lambda} \left|\frac{1}{h'(h^{-1}(q,p))}\right|^{-2s}\,\ud q \nonumber\\
&=\frac{1}{2^s c\lambda} \int_0^{c\lambda} |(h^{-1})'(q,p)|^{-2s}\,\ud q.
\end{align}
Since the fact that the function $h$ is a bijection and $h\neq0$, the derivative $h'$ of which  is analytic by relations $\eqref{1.4}$-$\eqref{1.5}$, $H(q,p)=(h^{-1})'(q,p))$ is analytic on $\Omega$. Consider that the map: $\mathcal S\rightarrow \mathbb C$, given by
\begin{equation*}
\mathcal H(z)=H(-\frac{i}{k}\log z).
\end{equation*}
According to the Morera's theorem, $\mathcal H(z)$ is analytic on $\mathcal S$. Furthermore, we find
\begin{align}\label{M}
\mu_s(E,-\frac{\log r}{k})&=\frac{1}{2^sc\lambda} \int_0^{c\lambda} |H(q,-\frac{\log r}{k})|^{-2s}\,\ud q\nonumber\\
&=\frac{1}{2\pi2^s}\int_0^{2\pi}|H(\frac{\theta}{k},-\frac{\log r}{k})|^{-2s}\,\ud \theta\nonumber\\
&=\frac{1}{2\pi2^s}\int_0^{2\pi}|H(\frac{\theta}{k}-i\frac{\log r}{k})|^{-2s}\,\ud \theta\nonumber\\
&=\frac{1}{2\pi2^s}\int_0^{2\pi}|H(-\frac{i}{k}\log(r\mathrm{e}^{i\theta}))|^{-2s}\,\ud \theta\nonumber\\
&=\frac{1}{2^s}{M}_{-2s}^{-2s}(\mathcal H,r).
\end{align}
Moreover, let $r=\mathrm{e}^{-kp}$, we find
\begin{equation*}
\mu_s(E,p)=\frac{1}{2^s}{M}_{-2s}^{-2s}(\mathcal H,\mathrm{e}^{-kp})
\end{equation*}
and
\begin{equation*}
\log\mu_s(E,p)=-2s\log{M}_{-2s}(\mathcal H,\mathrm{e}^{-kp})-s\log2.
\end{equation*}
Applying the Hardy's convexity theorem, for $-2s\geq 1$ and $r=\mathrm{e}^{-kp}<1$, $\log {M}_{-2s}(\mathcal H,r)$ is a convex function of $\log r$, and $ M_{-2s}(\mathcal H,r)$ is non-decreasing. The exact argument is analogous to the case of $s\geq\frac{1}{2}$. Therefore, it's easy find that $\log\mu_s(E,p)$ and $\mu_s(E,p)$ are convex function, which means the function
\begin{align}\label{3.7}
\frac{\ud}{\ud \, p}\mu_s(E,p)&=\frac{1}{2^sc\lambda} \int_0^{c\lambda} \frac{\ud}{\ud \,p}|H(q,p))|^{-2s}\,\ud q\nonumber\\
&=\frac{1}{2^sc\lambda} \int_0^{c\lambda} \frac{\ud}{\ud \,p}\left(\frac{1}{(u(h^{-1}(q,p)) - c)^2 + (v(h^{-1}(q,p)))^2 }\right)^{-2s} \,\ud q\nonumber\\.
\end{align}
is increasing.
Let $h^{-1}(q,p)=(x,y)$, according to the chain rule, we obtain
\begin{align}\label{3.8}
\frac{\partial}{\partial p}\{(u-c)^2+v^2\}=2(u-c)(u_x\frac{\partial x}{\partial p}+u_y\frac{\partial y}{\partial p})+2v(v_x\frac{\partial x}{\partial p}+v_y\frac{\partial y}{\partial p}).
\end{align}
For $p>0$, by $\eqref{1.7}$ and $\eqref{2.17}$, we have
\begin{align}\label{3.9}
\underset{p\rightarrow+\infty}\lim y(q,p)=-\infty
\end{align}
and
\begin{align}\label{3.10}
\underset{p\rightarrow+\infty}\lim v((x(q,p),y(q,p))=0.
\end{align}
Moreover, computing the inverse of the Jacobian matrix of $h$, we find
\begin{align}\label{3.11}
\underset{p\rightarrow+\infty}\lim\frac{\partial x}{\partial p}(q, p)=\underset{p\rightarrow+\infty}\lim-\frac{v}{(u-c)^2+v^2}(x(q,p),y(q,p))=0.
\end{align}
In light of $\eqref{3.7}$-$\eqref{3.11}$, we have
\begin{align*}
\underset{p\rightarrow+\infty}\lim \underset{q\in(0,c\lambda)}\sup\frac{\ud}{\ud \, p}\mu_s(E,p)=0.
\end{align*}
Therefore, for any $p>0$, then
\begin{align*}
\mu_s'(E,p)\leq 0.
\end{align*}
Consequently, the function $\mu_s(E,p)$ is non-increasing for $p\in(0, +\infty)$, which we complete the proof of the theorem.
\end{proof}
Indeed, by $\eqref{1.7}$ and $\eqref{2.17}$, we find
\begin{equation*}
\underset{p\rightarrow+\infty}\lim \underset{q\in(0,c\lambda)}\inf{\mu_s(E,p)}=\left(\frac{\sqrt{2}c}{2}\right)^{2s}.
\end{equation*}
\begin{theorem}\label{Theorem2.1}
For $s>0$ and $p>0$, the function $[\mu_s(E,p)]^{1/s}$ is convex and non-increasing, and $\log\mu_s(E,p)$ is convex.
\end{theorem}
\begin{proof}
The argument of this theorem is very similar to Theorem $\ref{Theorem3.1}$. Recall that the hodograph transform $h=q+ip$ is analytic on $D$, the derivative $h'$ of which is analytic and periodic, and $h'=u-c-iv$. From $\eqref{3.2}$, we find
\begin{equation}\label{a}
[\mu_s(E,p)]^{1/s}=\frac{1}{2} M_{2s}^2(G,\mathrm{e}^{-kp})
\end{equation}
and
\begin{equation}\label{b}
\log\mu_s(E,p)=2s\log M_{2s}(G,\mathrm{e}^{-kp})-s\log2.
\end{equation}
Let $r=\mathrm{e}^{-kp}$, the Hardy's convexity theorem for analytic function in a unit disk ensures that the function $M_{2s}(G,r)$ is non-decreasing and convex, and $\log M_{2s}(G,r)$ is convex. Moreover, let $\mathcal F(r)=M_{2s}(G,r)$, then we obtain
\begin{equation}\label{c}
\frac{\ud}{\ud p}\frac{1}{2} M_{2s}^2(G,\mathrm{e}^{-kp})=\frac{\ud}{\ud p}\frac{1}{2} \mathcal F^2(\mathrm{e}^{-kp})=-k\mathrm{e}^{-kp} \mathcal F'(\mathrm{e}^{-kp})\leq 0
\end{equation}
and
\begin{equation}\label{d}
\frac{\ud^2}{\ud p^2}\frac{1}{2} M_{2s}^2(G,\mathrm{e}^{-kp})=k^2\mathrm{e}^{-kp}(\mathrm{e}^{-kp} \mathcal F''(\mathrm{e}^{-kp})+ \mathcal F'(\mathrm{e}^{-kp})\geq0.
\end{equation}
In light of $\eqref{a}$-$\eqref{d}$, the function $[\mu_s(E,p)]^{1/s}$ is convex and non-increasing, and $\log\mu_s(E,p)$ is convex.
\end{proof}
Furthermore, from $\eqref{1.7}$ and $\eqref{2.17}$, we find
\begin{equation*}
\underset{p\rightarrow+\infty}\lim \underset{q\in(0,c\lambda)}\inf{[\mu_s(E,p)]^{1/s}}=\frac{c^2}{2}.
\end{equation*}
The physical motivation of Theorem $\ref{Theorem3.1}$ and $\ref{Theorem3.1}$ provide an important understanding of the structure of the kinetic energy of deep-water waves. The exploitation and utilization of ocean wave energy is becoming more and more important because of its vast but yet somewhat untapped potential. Studying the convexity and logarithmic convexity properties of integral of kinetic energy can provide theoretical basis for numerical simulations of the kinetic energy for irrotational deep-water waves.
\section{ The periodicity of fluid particle's trajectory}\label{sec3}
In this section, we derive some properties of the streamline time-period, namely the elapsed time
per period of the streamline, which refers to the time it takes to traverse one period in the moving plane(see \cite{r2}). In fact, there are no closed particles in a deep-water Stokes wave(see Proposition 3.2 in \cite{r5}). The periodicity of the trajectory of any fluid particle is ensured by the properties of Stokes waves. A lot of experimental evidences show that for waves which are not near the breaking or spilling state the speed of any fluid particle is generally appreciably smaller than the wave propagation speed \cite{r3}. There are thus no stagnation points throughout the flow. The above argument suggests the streamline coincides with the trajectory of particle. Moreover, from $\eqref{2.15}$, we find that for any point $(x,y)$ in the fluid domain $D$ there is a unique point $(q,p)$ correspondence in the open half rectangle domain $\Omega$. Since $h=q+ip$ is a bijection, there is a smooth streamline $y=\xi_p(x)$ in the fluid domain $D$ only corresponding to a horizontal line $p$ in the open half rectangle domain $\Omega$, that is
\begin{align}
\psi(x, \xi_p(x))=p.
\end{align}
According to the implicit differentiation rule, by $\eqref{1.9}$, then we obtain
\begin{align}\label{4.2}
\psi_x+\psi_y \xi'_p(x)=0\Rightarrow \xi'_p(x)=-\frac{\psi_x}{\psi_y}=\frac{\varphi_y}{\varphi_x}=\frac{v}{u-c}.
\end{align}
In light of $\eqref{1.8}$ and $\eqref{uj}$, we have
\begin{align*}
\xi'_p(x)\geq0~~\text{for}~x\in [-L/2,0] \quad \text{and} \quad \xi_p'(x) \leq 0~~\text{for}~x\in [0,L/2].
\end{align*}
The motion of fluid particles follows the following differential systems:
\begin{equation}\label{4.4}
\begin{cases}
x'(t)=u(x(t)-ct,y(t)),\\
y'(t)=v(x(t)-ct,y(t)),
\end{cases}
\end{equation}
with initial data $(x(t_0),y(t_0))=(x_0,y_0)$. Due to the boundedness and smoothness of the velocity field $(u,v)$, there is a unique global solution $(x(t; x_0,t_0), y(t; x_0,t_0))$ that depends on initial data $(x_0,y_0)$ and $t$. Indeed, each particle describes the repeating pattern. Therefore for any particle path repeats the same trajectory over a time, that is the streamline time-period $\mathcal T$.
In the next section, we will show the existence of $\mathcal T$, and investigate whether it is independent of the initial data $(x_0,t_0)$.
\begin{theorem}\label{Theorem4.1}
Suppose $(q,p)\in(0,c\lambda)\times(0,+\infty)$ and $\lambda>0$. Passing to the moving frame, the equation
\begin{align*}
x(\mathcal T+t_0;t_0,x_0)-c\mathcal T=x_0-\lambda
\end{align*}
has a unique solution $\mathcal T(p)>0$ which depends solely on the streamline $p$ and is independent of initial data $(x_0,t_0)$.
Moreover, the function $\mathcal T(p)$  non-increasing and convex, and $\log \mathcal T(p)$ is a convex. For $p\in\mathbb R$, $\mathcal T(p)$ is an even function.
\end{theorem}
\begin{proof}
Define the function
\begin{align*}
g(t)=x(t+t_0;t_0,x_0)-ct-x_0+\lambda.
\end{align*}
Furthermore,  we obtain
\begin{align*}
g'(t)=u(x,y)-c\leq-\delta_0,\quad g(0)=\lambda>0,
\end{align*}
where $\delta_0 =\underset{(x,y)\in D}{\min} \{c-u(x,y)\} > 0$. According to the Lagrange Mean-Value theorem, we have
\begin{align*}
g(t)-g(0)=g'(\tau)t\Rightarrow g(t)=g'(\tau)t+\lambda\leq-\delta_0 t +\lambda\quad \text{ for some } \tau\in(0,t).
\end{align*}
Taking advantage of geometric properties, we find that there is $\mathcal T>0$ with $g(\mathcal T)=0$.
Consider that the function
\begin{align*}
\zeta(s)=x(s;t_0, x_0)-cs,
\end{align*}
then we obtain
\begin{align*}
\zeta'(s)=u(\zeta(s),\xi_p(\zeta(s))-c\text{,}\quad\zeta(t_0)=x_0.
\end{align*}
Moreover, we find
\begin{align}\label{B}
\zeta(\mathcal T+t_0)=x(\mathcal T+t_0;t_0, x_0)-c(\mathcal T+t_0)=x_0-ct_0-\lambda.
\end{align}
For $s\in[t_0,\mathcal T+t_0]$, we obtain
\begin{align*}
\mathcal T=\int_{t_0}^{\mathcal T+t_0}\frac{\zeta'(s)}{u(\zeta(s),\xi_p(\zeta(s)))-c}\,\ud s =\int_{\zeta(t_0)}^{\zeta(\mathcal T +t_0)}\frac{1}{u(\zeta(s),\xi_p(\zeta(s))-c}\,\ud \zeta(s).
\end{align*}
Moreover, by means of the periodicity of $u$, we have
\begin{align*}
\mathcal T=\int_{x_0-ct_0}^{x_0-ct_0-\lambda}\frac{1}{u(\zeta,\xi_p(\zeta))-c}\,\ud \zeta=\int_{-\frac{\lambda}{2}}^{\frac{\lambda}{2}}\frac{1}{c-u(x,\xi_p(x))}\,\ud x.
\end{align*}
Therefore, $\mathcal T$ depends only on the horizontal line $p$ in the open half rectangle domain $\Omega$, and is independent of initial data $(x_0,t_0)$.
Note that the function $h=q+ip=\varphi+i\psi$ is analytic on $D$, the derivative $h'$ of which is analytic on $D$.  By means of the hodograph transform $h$, we consider
\begin{align*}
q=q_p(x)=\varphi(x, \xi_p(x)).
\end{align*}
According to the implicit differentiation rule, then
\begin{align}\label{4.10}
q'_p(x)=\varphi_x+ \varphi_y \xi_p'(x).
\end{align}
In light of $\eqref{1.9}$ and $\eqref{4.2}$, we obtain
\begin{align}\label{4.11}
q'(x)=\frac{\varphi_x^2+\varphi_y^2}{\varphi_x}=\frac{(u-c)^2+v^2}{u-c}.
\end{align}
Since $E=((u-c)^2 +v^2)/2$, then we have
\begin{align}\label{4.12}
q'(x)=\frac{2E}{u(x,\xi_p(x))-c}.
\end{align}
For $x\in[-\lambda/2,\lambda/2]$, since $q(-\lambda/2)=c\lambda$ and $q(\lambda/2)=0$, in light of $\eqref{4.10}$-$\eqref{4.12}$ we obtain
\begin{align}\label{4.13}
\mathcal T(p)=\int_{-\frac{\lambda}{2}}^{\frac{\lambda}{2}}\frac{1}{c-u(x,f_p(x))}\,\ud x &=-\int_{-\frac{\lambda}{2}}^{\frac{\lambda}{2}}\frac{q'(x)}{2E(x,f_p(x))}\,\ud x
&=\int_{0}^{c\lambda}\frac{1}{2E(h^{-1}(q,p))}\,\ud q.
\end{align}
According to Theorem $\ref{Theorem3.1}$, for the index $s=-1$, we have
\begin{align*}
\mathcal T(p)=\frac{c\lambda}{2}\mu_{-1}(E,p).
\end{align*}
Therefore, $\mathcal T(p)$ is a convex and non-increasing function, and $\log \mathcal T(p)$ is convex.
For $p\in\mathbb R$, we find that
\begin{align*}
\mathcal T(-p)=\int_{0}^{c\lambda}\frac{1}{2E(h^{-1}(q,-p))}\,\ud q=\int_{0}^{c\lambda}\frac{1}{|h'(h^{-1}(q,-p))|^2}\,\ud q=\mathcal T(p),
\end{align*}
which implies $\mathcal T(p)$ is an even function.
\end{proof}
\begin{theorem}\label{Remark4.1}
The streamline time-period $\mathcal T(p)$ of the fluid particle path satisfies the inequality
\begin{equation*}
\frac{\lambda}{c} < \mathcal T(p)\leq\int_{-\frac{\lambda}{2}}^{\frac{\lambda}{2}}\frac{1}{c-u(x,\eta(x))}\,\ud x \quad \text{ for any }\quad 0\leq p<+\infty,
\end{equation*}
where $\lambda/c$ is the wave period. Moreover, there are no closed paths for all fluid particles and the drift of any streamline is positive, that is
\begin{equation*}
x(\mathcal T+t_0)-x(t_0)=c\mathcal T-\lambda>0.
\end{equation*}
\end{theorem}
\begin{proof}
For $(q,p)\in(0,c\lambda)\times[0,+\infty)$, we rewrite $\eqref{4.13}$ as
\begin{align*}
\mathcal T(p)=\int_{-\frac{\lambda}{2}}^{\frac{\lambda}{2}}\frac{1}{c-u(x,\xi_p(x))}\,\ud x.
\end{align*}
According to Theorem $\ref{Theorem4.1}$, $\mathcal T(p)$ is non-increasing for $p\in[0,+\infty)$. Since we assume that the velocity field $(u,v)$ is bounded and smooth in the fluid domain $D$,  the function $\mathcal T(p)$ is smooth and bounded by relation $\eqref{4.13}$. Therefore, the supremum of $\mathcal T(p)$ is attained at $p=0$, namely the free surface $y=\eta(x)$, and the infimum of $\mathcal T(p)$ is attained at $p=+\infty$. From $\eqref{1.10}$, then $p=0$ on the free surface $y=\eta(x)$. In light of $\eqref{1.7}$, $\eqref{1.8}$ and $\eqref{2.17}$, then $u\rightarrow 0$ as $p\rightarrow+\infty$. Furthermore, we have
\begin{align*}
\underset{(q,p)\in\Omega }\sup \mathcal T(p)=\underset{p\rightarrow 0 }\lim \mathcal T(p)=\int_{-\frac{\lambda}{2}}^{\frac{\lambda}{2}}\frac{1}{c-u(x,\eta(x))}\,\ud x,
\end{align*}
and
\begin{align*}
\underset{(q,p)\in\Omega }\inf \mathcal  T(p)=\underset{p\rightarrow +\infty }\lim \mathcal T(p)=\underset{p\rightarrow +\infty }\lim \int_{-\frac{\lambda}{2}}^{\frac{\lambda}{2}}\frac{1}{c-u(x,\xi_p(x))}\,\ud x=\frac{\lambda}{c}.
\end{align*}
For $0\leq p <+\infty$, then
\begin{align*}
\frac{\lambda}{c} < \mathcal T(p)\leq\int_{-\frac{\lambda}{2}}^{\frac{\lambda}{2}}\frac{1}{c-u(x,\eta(x))}\,\ud x.
\end{align*}
By the definition of $\mathcal T$, we have
\begin{equation*}
x(\mathcal T+t_0)-x(t_0)=c\mathcal T-\lambda>0.
\end{equation*}
According to Lemma 3.1 in \cite{r5}, we can infer that there are no closed paths for all fluid particles, and the drift of any streamline is positive.
\end{proof}
\begin{remark}
Theorem $\ref{Remark4.1}$ ensures any fluid particle's path is not closed in deep-water Stokes flow, which is identical to the result in \cite{r5}.  Moreover, there is  a closed particle path if and only if $\mathcal T=\lambda/c$.
\end{remark}
\section{The total kinetic energy of a fluid particle }\label{sec4}
In this section, we discuss some qualitative results about the total kinetic energy over a streamline time-period for a fluid particle.
The total kinetic energy of a fluid particle located initially at $(x_0,y_0)$ with $y_0=\xi_p(x_0)$ over a streamline time-period is given by
\begin{align*}
\mathcal E( p,x_0)=\frac{1}{2}\int_0^{\mathcal T(p)}[(x'(t;x_0))^2+(y'(t;x_0))^2]\,\ud t.
\end{align*}
Passing to the moving frame, the total kinetic energy of a fluid particle located initially at $(x_0,y_0)$ with $y_0=\xi_p(x_0)$  over a streamline time-period is given by
\begin{align*}
\mathbb E(p,x_0)=\frac{1}{2}\int_0^{\mathcal T(p)}[(x'(t;x_0)-c)^2+(y'(t;x_0))^2]\,\ud t.
\end{align*}
\begin{theorem}\label{Theorem5.1}
Suppose $t\in[0,\mathcal T(p)]$ and $(q,p)\in(0,c\lambda)\times(0,+\infty)$. The function $\mathcal E(p,x_0)$ is a convex and non-increasing, which  depends solely on the streamline $p=\psi$ and is independent of initial location $x_0$. Moreover, $\log \mathcal E(p,x_0)$ is a convex function. The function $\mathbb E(p,x_0)$ is equal to a constant $c\lambda/2$.
\end{theorem}
\begin{proof}
Note that the hodograph transform $h=q+ip=\varphi+i\psi$ is analytic in the interior of $D$, and $h'=u-c+i(-v)$.
For $t\in[0,\mathcal T(p)]$, we have
\begin{align*}
\mathcal E( p,x_0)&=\frac{1}{2}\int_0^{\mathcal T(p)}[(x'(t;x_0))^2+(y'(t;x_0))^2]\,\ud t\\
&=\frac{1}{2}\int_0^{\mathcal T(p)}[((x(t;x_0)-ct)'+c)^2+(y'(t;x_0))^2]\,\ud t\\
&=\frac{1}{2}\int_0^{\mathcal T(p)}|h'(x(t;x_0)-ct, y'(t;x_0))+c|^2\,\ud t\\
&=\frac{1}{2}\int_0^{\mathcal T(p)}|h'(x(t;x_0)-ct, \xi_p'(x(t;x_0)-ct))+c|^2\,\ud t.
\end{align*}
Let $\Theta(t)=x(t;x_0)-ct$, then $\ud \Theta=(u(x(t;x_0)-ct, y(t;x_0))-c)\ud t$. Moreover, we obtain
\begin{equation*}
\Theta(0)=x_0, \quad \Theta(\mathcal T(p))=x(\mathcal T;x_0)-cT(p)=x_0-\lambda.
\end{equation*}
Since $u$ and $h'$ are periodic in the $x$-variable, we have
\begin{align}\label{5.4}
\mathcal E(p,x_0)&=\frac{1}{2}\int_{x_0}^{x_0-\lambda}|h'(\Theta,\xi_p(\Theta))+c|^2 \frac{1}{u(\Theta,\xi_p(\Theta))-c}\,\ud \Theta\nonumber\\
&=\frac{1}{2}\int_{-\frac{\lambda}{2}}^{\frac{\lambda}{2}}|h'(x,\xi_p(x))+c|^2 \frac{1}{c-u(x,\xi_p(x))}\,\ud x.
\end{align}
Consequently, we observe that the function $\mathcal E$  depends solely on the streamline $p=\psi$, and is independent of $x_0$.
In light of $\eqref{4.10}$ and $\eqref{4.11}$, we have
\begin{align}\label{5.5}
\mathcal E(p)&=\frac{1}{2}\int_{-\frac{\lambda}{2}}^{\frac{\lambda}{2}}|h'(x,\xi_p(x))+c|^2 \frac{1}{c-u(x,\xi_p(x))}\,\ud x\nonumber\\
&=\frac{1}{2}\int_{0}^{c\lambda}\left|\frac{h'+c}{h'}\right|^2(h^{-1}(q,p)) \,\ud q\nonumber\\
&=\frac{1}{2}\int_{0}^{c\lambda}\left|1+c(h^{-1})'(q,p)\right|^2 \,\ud q.
\end{align}
When $c=0$, we can infer that $\mathbb E(p)=\mathcal E(p)=c\lambda/2$.
Note that $ J(q,p)=(1+c(h^{-1})')(q,p)$ is analytic on $\Omega$. Consider that the map $\mathcal S\rightarrow \mathbb C$, given by
\begin{equation*}
\mathcal J(z)=J(-\frac{i}{k}\log z),
\end{equation*}
where
\begin{align*}
\mathcal{S} = \left\{z \in \mathbb{C}: |z| < 1\right\}.
\end{align*}
According to the Morera's theorem, $\mathcal J(z)$ is analytic on $\mathcal S$.
Let $\theta=kq$ and $r=\mathrm{e}^{-kp}$, where $k=2\pi/c\lambda$, we find
\begin{align*}
\mathcal E(-\frac{\log r}{k})&=\frac{1}{2k}\int_{0}^{2\pi}\left|J(\frac{\theta}{k},-\frac{\log r}{k})\right|^2 \,\ud \theta\\
&=\frac{1}{2k}\int_{0}^{2\pi}\left|J(\frac{\theta}{k}-i\frac{\log r}{k})\right|^2 \,\ud \theta\\
&=\frac{1}{2k}\int_{0}^{2\pi}\left|J(-\frac{i}{k}\log(r\mathrm{e}^{i\theta}))\right|^2 \,\ud \theta\\
&=\frac{c\lambda}{2}M_2(\mathcal J,r).
\end{align*}
Let $r=\mathrm{e}^{-kp}$, we have
\begin{equation*}
\mathcal E(p)=\frac{c\lambda}{2}M_2(\mathcal J,\mathrm{e}^{-kp}).
\end{equation*}
According to the Hardy's convexity theorem, for $r<1$ the function $M_2(\mathcal J,r)$ is non-decreasing, and $\log M_2(\mathcal J,r)$ is convex.
The exact same argument as the one used in the proof of Theorem $\ref{Theorem3.1}$ shows that $ \mathcal E(p)$ is a convex and non-increasing function. Moreover, $\log \mathcal E(p)$ is a convex function.
\end{proof}
 Moreover, in light of $\eqref{4.12}$, we have
\begin{align}\label{5.6}
\mathcal E(p)=\frac{1}{2}\int_{0}^{c\lambda}\frac{E_0}{E}(h^{-1}(q,p)) \,\ud q,
\end{align}
where $E_0=(u^2+v^2)/2$ is kinetic energy for a fluid particle at different depths. By $\eqref{3.1}$ and $\eqref{5.6}$, we have
\begin{align*}
\mathcal E(p)=\frac{c\lambda}{2} \mu_1(E_0\cdot E^{-1},p).
\end{align*}
Therefore, for $p\in(0,+\infty)$, the function $ \mu_1(E_0\cdot E^{-1},p)$ is a convex and non-increasing. Moreover, $\log\mu_1(E_0\cdot E^{-1},p)$ is a convex function.
\begin{remark}
For $p\in(0,+\infty)$, Theorem $\ref{Theorem5.1}$ shows that the total kinetic energy $\mathcal E(p)$ over a streamline time-period $\mathcal T(p)$  decreases with the elevation of the streamline $p$, and attains its maximum on the free surface.
\end{remark}
Now, we consider extending Theorem $\ref{Theorem5.1}$ to the general case.
\begin{theorem}\label{Theorem5.2}
Suppose $s\in(-\infty, -\frac{1}{2}]\cup [\frac{1}{2}, +\infty)\cup\{0\}$, $t\in[0,\mathcal T(p)]$ and $(q,p)\in(0,c\lambda)\times(0,+\infty)$. The function $\mathcal E_s(p,x_0)$ given by
\begin{align*}
\mathcal E_s( p,x_0)=\frac{1}{2}\int_0^{\mathcal T(p)}[(x_p'(t;x_0))^2+(y_p'(t;x_0)^2]^s\,\ud t
\end{align*}
is convex and non-increasing, which  depends solely on the streamline $p=\psi$ and is independent of $x_0$. Moreover, the function $\log\mathcal E_s(p,x_0)$ is  convex. For $s\in (-\infty, \frac{1}{2}]\cup [\frac{3}{2}, +\infty)$, passing to the moving frame, the function
\begin{align*}
\mathbb E_s(p,x_0)=\frac{1}{2}\int_0^{\mathcal T(p)}[(x_p'(t;x_0)-c)^2+(y_p'(t;x_0))^2]^s\,\ud t
\end{align*}
is convex and non-increasing, which  depends solely on the streamline $p=\psi$ and is independent of $x_0$. Moreover, the function $\log\mathbb E_s(q,p)$ is convex.
\end{theorem}
\begin{proof}
The proof this theorem is similar to Theorem $\ref{Theorem5.1}$. Therefore, we only sketch some of main steps.
In light of $\eqref{5.4}$-$\eqref{5.5}$, we have
\begin{align}
\mathcal E_s( p,x_0)&=\frac{1}{2}\int_{-\frac{\lambda}{2}}^{\frac{\lambda}{2}}|h'(x,\xi_p(x))+c|^{2s} \frac{1}{c-u(x,\xi_p(x))}\,\ud x \nonumber\\
&=\frac{1}{2}\int_{0}^{c\lambda}\frac{|h'+c|^{2s}}{|h'|^2}(h^{-1}(q,p)) \, \ud q.\nonumber
\end{align}
Let $c=0$, we have
\begin{align}
\mathbb E_s(p,x_0)=\frac{1}{2}\int_{0}^{c\lambda}|h'|^{2(s-1)}(h^{-1}(q,p)) \, \ud q.\nonumber
\end{align}
In light of $\eqref{4.12}$, we obtain
\begin{align}\label{5.10}
\mathbb E_s(p,x_0)=2^{s-2}\int_{0}^{c\lambda}E^{s-1}(h^{-1}(q,p)) \,\ud q
\end{align}
and
\begin{align}\label{5.11}
\mathcal E_s(p,x_0)=2^{s-2}\int_{0}^{c\lambda}\frac{E_0^s}{E}(h^{-1}(q,p)) \,\ud q,
\end{align}
where $E_0=(u^2+v^2)/2$ and $E=((u-c)^2+v^2)/2$.
By $\eqref{5.10}$ and $\eqref{5.11}$, then we find that $\mathcal E_s$  and $\mathbb E_s$  depend solely on the streamline $p=\psi$, and are independent of $x_0$.
For $s-1\in(-\infty,-\frac{1}{2}]\cup[\frac{1}{2},+\infty)$, from $\eqref{3.1}$, we obtain
\begin{align*}
\mathbb E_s(p)=2^{s-2}c\lambda \cdot\mu_{s-1}(E,p).
\end{align*}
Therefore, for $p\in(0,+\infty)$ and $s\in(-\infty,\frac{1}{2}]\cup[\frac{3}{2},+\infty)$, according to Theorem $\ref{Theorem3.1}$, the function $\mathbb E_s(p)$ is convex and non-increasing, and $\log \mathbb E_s(p)$ is convex.
For the function $\mathcal E_s(p)$, when $s=0$, we get
\begin{align*}
\mathcal E_0(p)=\frac{\mathcal T(p)}{2}.
\end{align*}
According to Theorem $\ref{Theorem4.1}$, for $p\in(0,+\infty)$, the function $\mathcal E_s(p)$ is convex and non-increasing, and $\log \mathcal E_s(p)$ is convex.
For $s\in(-\infty,-\frac{1}{2}]\cup[\frac{1}{2},+\infty)$, we have
\begin{align*}
\mathcal E_s(p)=2^{s-2}c\lambda \cdot\mu_1(E_0^{s}\cdot E^{-1},p).
\end{align*}
If we let $s=m_1/m_2$, then we find
\begin{align*}
\mathcal E_s(p)=\frac{1}{2}\int_{0}^{c\lambda}\left|\frac{(h'+c)^{m_1}}{(h')^{m_2}}(h^{-1}(q,p))\right|^{\frac{2}{m_2}} \, \ud q.
\end{align*}
Similar to the argument in Theorem $\ref{3.1}$ and Theorem $\ref{Theorem5.1}$, applying the Hardy's convexity theorem, the function $\mathcal E_s(p)$ is convex and non-increasing, and $\log \mathcal E_s(p)$ is also convex.
\end{proof}

\section*{Data availability }
No data was used for the research described in the article.

\section*{Conflict of interest }
The authors do not have any other competing interests to declare.

\end{document}